\documentclass[11pt,a4paper,reqno]{amsart}
\usepackage[english]{babel}
\usepackage[T1]{fontenc}
\usepackage{verbatim}
\usepackage{palatino}
\usepackage{amsmath}
\usepackage{mathabx}
\usepackage{amssymb}
\usepackage{amsthm}
\usepackage{amsfonts}
\usepackage{graphicx}
\usepackage{esint}
\usepackage{color}
\usepackage{mathtools}
\usepackage{overpic}
\usepackage{adjustbox}

\usepackage[colorlinks,linkcolor=blue,citecolor=blue]{hyperref}
\title{On bounded energy of convolution of fractal measures}
\author[G. Yi]{Guangzeng Yi}
\address{$\ast$Department of Mathematics and Statistics\\ University of Jyv\"askyl\"a,
	P.O. Box 35 (MaD)\\
	FI-40014 University of Jyv\"askyl\"a\\
	Finland}
\email{guangzeng.m.yi@jyu.fi}
\date{\today}
\subjclass[2020]{28A80 (primary) 05B99, 05D99, 51A20 (secondary)}
\keywords{Incidences, Riesz energy, Fourier decay}
\thanks{G.Y. is supported by the European Research Council (ERC) under the European Union’s Horizon Europe research and innovation programme (grant agreement No 101087499). }

\newcommand{\R}{\mathbb{R}}

\newcommand{\N}{\mathbb{N}}

\def\Barint_#1{\mathchoice
	{\mathop{\vrule width 6pt height 3 pt depth -2.5pt
			\kern -8pt \intop}\nolimits_{#1}}%
	{\mathop{\vrule width 5pt height 3 pt depth -2.6pt
			\kern -6pt \intop}\nolimits_{#1}}%
	{\mathop{\vrule width 5pt height 3 pt depth -2.6pt
			\kern -6pt \intop}\nolimits_{#1}}%
	{\mathop{\vrule width 5pt height 3 pt depth -2.6pt
			\kern -6pt \intop}\nolimits_{#1}}}

\numberwithin{equation}{section}

\theoremstyle{plain}
\newtheorem{thm}[equation]{Theorem}
\newtheorem*{"thm"}{"Theorem"}
\newtheorem{conjecture}[equation]{Conjecture}
\newtheorem{lemma}[equation]{Lemma}
\newtheorem{sublemma}[equation]{Sublemma}
\newtheorem{ex}[equation]{Example}
\newtheorem{cor}[equation]{Corollary}
\newtheorem{proposition}[equation]{Proposition}
\newtheorem{question}{Question}

\theoremstyle{definition}

\newtheorem{definition}[equation]{Definition}

\theoremstyle{remark}
\newtheorem{remark}[equation]{Remark}

\newcommand{\nref}[1]{(\hyperref[#1]{#1})}

\DeclareMathSymbol{\intop} {\mathop}{mathx}{"B3}

\begin{document}
	
	\begin{abstract}
		For all $s\in[0,1]$ and $t\in(0,s]\cup [2-s,2)$, we find the supremum of numbers $\omega\in(0,2)$ such that
		\[\textup{I}_\omega(\mu\ast\sigma) \lesssim 1,\]
		where $\mu$ is any Borel measure on $B(1)$ with $\textup{I}_t(\mu)\leq 1$ and $\sigma$ is any $(s,1)$-Frostman measure on a $C^2$-graph with non-zero curvature. As an application, we use this to show the sharp $L^6$-decay of Fourier transform of $\sigma$ when $s\in [\frac{2}{3}, 1]$.
	\end{abstract}
	
	\maketitle
	\setcounter{tocdepth}{1}
	\tableofcontents
	
	\section{Introduction}
	Assume that $\psi\in C^2(\mathbb R)$ satisfies the positive curvature condition
	\begin{equation}\label{curvature assumption}
		\psi''(x)>0,\quad x\in\mathbb R.
	\end{equation}  
	Let $\Gamma$ be the truncated graph of $\psi$ on $[-1,1]$. For $0<\omega<d$, let $\textup{I}_\omega(\mu)$ be the $\omega$-dimensional Riesz energy of a measure $\mu$, which is defined as
	\begin{equation}\label{Riesz}
		\textup{I}_\omega(\mu):=\int_{\mathbb R^{2d}} \frac{d\mu (x) d\mu (y)}{|x-y|^\omega}=c(\omega,d)\int_{\mathbb R^d} |\hat{\mu}(\xi)|^2\cdot|\xi|^{\omega-d} d\xi.
	\end{equation}
	See \cite[Theorem 3.10]{MR3617376} for the proof of the second identity. We also recall the definition of Frostman measures.
	\begin{definition}
		Let $u\in[0,2]$ and $C_\mu\geq 1$. A Borel measure $\mu$ on $\mathbb R^2$ is called a $(u,C_\mu)$-Frostman measure if $\mu(B(x,r))\leq C_\mu r^u$ for all $x\in \mathbb R^2$ and $r>0$. Here $C_\mu$ is also called the Frostman constant of $\mu$.
	\end{definition}
	
	Of concern is the following problem.
	\begin{question}\label{qt}
		Given $s\in[0,1]$ and $t\in(0,2)$, let $\mu$ be a Borel measure on $B(1)$ such that $\textup{I}_t(\mu)\leq1$ and let $\sigma$ be an $(s, 1)$-Frostman measure on $\Gamma$. What is the supremum $\mathfrak{f}(s,t)$ of $\omega\in(0,2)$ such that $\textup{I}_{\omega}(\mu\ast\sigma)\lesssim_{\psi,s,t,\omega} 1$ for all such $\mu$ and $\sigma$? 
	\end{question}
	
	The main result in this paper answers Question \ref{qt} partially. As one can see from \eqref{form-main}, the cases $t\in [s,3s], s\in(0,\frac{1}{2}]$ and $t\in [s, 2-s], s\in [\frac{1}{2},1]$ remain unsolved.
	\begin{thm}\label{main}
		For $s\in[0,1]$ and $t\in(0,2)$, we have 
		\begin{equation}\label{form-main}
			\mathfrak{f}(s,t)=
			\begin{cases}
				s+t, & \text{when} ~~ t\in(0,s],~s\in (0, 1],\\
				s+1, & \text{when}~~ t\in [2-s, s+1],~s\in [\frac{1}{2}, 1],\\
				t,   & \text{when}~~ t\in [3s,s+1],~s\in [0, \frac{1}{2}],\\
				t,   & \text{when}~~ t\in [s+1,2),~s\in [0, 1).
			\end{cases}
		\end{equation}
		Moreover, if $\textup{I}_t(\mu)<+\infty$ and $\sigma$ is an $(s,C_\sigma)$-Frostman measure on $\Gamma$, then for $\mathfrak{f}(s,t)=s+t$ or $\mathfrak{f}(s,t)=s+1$, there exists a constant $C=C(\psi, s,t,\epsilon)>0$ such that
		\begin{equation}\label{form-main2}
			\textup{I}_{\mathfrak{f}(s,t)-\epsilon}(\mu\ast\sigma)\leq C \big(\max\{\textup{I}_t(\mu), 1\}\big)\cdot C_\sigma^2,\quad \forall \epsilon\in(0,1).
		\end{equation}
	\end{thm}
	The following examples show that $\mathfrak{f}(s,t)$ cannot exceed those constants in \eqref{form-main}. Also, since we have the trivial estimate $\textup{I}_t(\mu\ast\sigma)\lesssim \textup{I}_t(\mu)$, the threshold "$t$" will be proven once the examples have been found. Therefore, the main part of this paper will be intended to prove the inequalities
	\begin{equation}\label{form-main3}
		\mathfrak{f}(s,t)\geq
		\begin{cases}
			s+t, & \text{when} ~~ t\in(0,s],~s\in (0, 1],\\
			s+1, & \text{when}~~ t\in [2-s, s+1],~s\in [\frac{1}{2}, 1].
		\end{cases}
	\end{equation}
	
	\subsection{Examples}\label{sub-1.1} 
	In the sequel, "$\textup{dim}_\text{H}$" denotes Hausdorff dimension and "$\textup{dim}_\text{B}$" denotes box counting dimension. We will use the basic fact for Borel sets $A\subset\mathbb R^2$:
	\[\textup{dim}_\text{H}(A)=\sup\{v: \textup{there is~} \mu\in\mathcal{M}(A) \text{~such that~} \textup{I}_v(\mu)<+\infty\},\]
	where $\mathcal{M}(A)$ is the family of positive finite Borel measures with compact support in $A$. The proof of this identity can be found in \cite[Theorem 2.8]{MR3617376}. 
	
	\begin{ex}[Case $\mathbf{t\in(0, s]}$ and $\mathbf{s\in(0, 1]}$]\label{ex-1} Let $\epsilon\in (0,1)$. Choose a measure $\mu$ on $B(1)$ with
		\[\mu(B(x, r))\sim r^{t+\epsilon},\quad x\in \textup{spt}(\mu),~r>0.\]
		Hence $\textup{I}_t(\mu)\lesssim 1$. Next, choose $\sigma$ on $\Gamma$ such that $\sigma(B(x, r))\sim r^{s}$ for any $x\in\textup{spt}(\sigma)$ and $r>0$. Since both $\mu$ and $\sigma$ are Ahlfors-regular, $\textup{dim}_{B}(spt(\sigma))=s$ and $\textup{dim}_{B}(spt(\mu))=t+\epsilon$, see \cite[Theorem 5.7]{MR1333890}. If ~$\textup{I}_\omega(\mu\ast\sigma)\lesssim_{\psi,s,t,\omega}1$, we deduce
		\[\begin{split}
			\omega&\leq \textup{dim}_\text{H}(spt(\mu\ast\sigma))\leq \textup{dim}_\text{H}(spt(\mu)+spt(\sigma))\\
			&\leq  \textup{dim}_{B}(spt(\sigma))+\textup{dim}_{B}(spt(\mu))=s+t+\epsilon.
		\end{split}\]
		By sending $\epsilon\to 0$, we conclude that the supremum $\mathfrak{f}(s,t)\leq s+t$.
	\end{ex}
	
	\begin{ex}[Case $\mathbf{t\in[2-s, s+1]}$ and $\mathbf{s\in[\tfrac{1}{2}, 1]}$]\label{ex-2} 
		We may assume $s<1$ since otherwise $s+1=2\geq \mathfrak{f}(s,t)$ is obvious. Let $A\subset [0,1]$ be an arbitrary $(t-s)$-dimensional subset. Next, by \cite[Example 7]{MR2742987}, for $\epsilon>0$ with $s+\epsilon<1$ we can find a compact subset $B\subset [0,1]$ such that
		\[\textup{dim}_{B}(B)=\textup{dim}_\text{H}(B)=\textup{dim}_\text{H}(B+B)=s+\epsilon.\]
		Then $\textup{dim}_\text{H}(A\times B)=t+\epsilon$ and $\textup{dim}_\text{H}(\Gamma\cap (B\times \mathbb R))=s+\epsilon$. Let $\mu$ be a finite Borel measure supported on $(A\times B)$ with $\textup{I}_t(\mu) \leq1$ and let $\sigma$ be an $(s,1)$-Frostman measure supported on $\Gamma\cap (B\times \mathbb R)$. If ~$\textup{I}_\omega(\mu\ast\sigma)\lesssim_{\psi,s,t,\omega}1$, we can deduce
		\[\begin{split}
			\omega&\leq \textup{dim}_\text{H}(spt(\mu\ast\sigma))\leq \textup{dim}_\text{H}(spt(\mu)+spt(\sigma))\\
			&\leq \textup{dim}_\text{H}((B+B)\times(A+\mathbb R))=s+1+\epsilon.
		\end{split}\]
		By sending $\epsilon \to 0$, we conclude that $\mathfrak{f}(s,t)\leq s+1$.
	\end{ex}
	
	\begin{ex}[Case $\mathbf{t\in[3s, s+1]}$ and $\mathbf{s\in [0,\tfrac{1}{2}]}$]\label{ex-3}
		Fix $t\in [3s,s+1)$ and $s\in[0,\tfrac{1}{2}]$ and note the case $t=s+1$ is contained in Example \ref{ex-4}. Let $\tau\in(t, \tfrac{3}{2}]$. We also choose $\psi(x)=x^2$ and write $G(x)=(x,x^2)$.
		
		\textbf{Step 1: constructing supporting sets.} Let $A=[0,1]\cap (\delta^{\frac{\tau}{3}}\mathbb Z)$, then the $\delta$-neighborhood of $A$ denoted by $A(\delta)$ is a union of $\sim \delta^{-\frac{\tau}{3}}$ equally spaced open intervals of length $2\delta>0$ with center points in arithmetic progression. We will need the explicit form: 
		\[A(\delta)=\{I_k:1\leq k \leq n,~n\sim \delta^{-\frac{\tau}{3}}\},~~\text{where}~~ I_k=(k\delta^{\frac{\tau}{3}}-\delta,k\delta^{\frac{\tau}{3}}+\delta).\]
		Since $\tau\leq \tfrac{3}{2}$, we similarly define $B=[0,1]\cap (\delta^{\frac{2\tau}{3}}\mathbb Z)$ and $B(\delta)=\{J_k:1\leq k \leq m,~m\sim \delta^{-\tfrac{2\tau}{3}}\}$, where $J_k=(k\delta^{\frac{2\tau}{3}}-\delta,k\delta^{\frac{2\tau}{3}}+\delta)$. A first observation is that
		\begin{equation}\label{form-observation}
			|A(\delta)+A(\delta)|_{\delta}\sim \delta^{-\tfrac{\tau}{3}},~~|B(\delta)+B(\delta)|_{\delta}\sim \delta^{-\tfrac{2\tau}{3}}~~\text{and}~~|A(\delta)\times B(\delta)|_{\delta}\sim \delta^{-\tau}.
		\end{equation}  
		Since $s<\frac{\tau}{3}$, we may assume that $\delta^s=j\cdot\delta^{\frac{\tau}{3}}+\bar{\delta}$ with $0\leq\bar{\delta}<\delta^{\frac{\tau}{3}}< \delta^s\leq1$, where $j=j(s,\tau)\geq1$ is an integer. Hence we can choose $D\subset A$ to be
		\[D:=\{{kj}\delta^{\tfrac{\tau}{3}}, 1\leq k\leq m\},~~~\text{where}~~m\sim \frac{\delta^{-\frac{\tau}{3}}}{j} = \frac{1}{\delta^s-\bar{\delta}}\sim \delta^{-s}.\]
		Then $G(D(\delta))\subset A(\delta)\times \psi(A(\delta))$ and $|G(D(\delta))|_\delta \sim \delta^{-s}$. So far we have finished the construction of supporting sets, namely, we will construct $\mu$ and $\sigma$ such that $\text{spt}(\mu)\subset A(\delta)\times B(\delta)$ and $\text{spt}(\sigma)\subset G(D(\delta))$.
		Before that, let us show the covering estimate
		\begin{equation}\label{form-coverofsum}
			\left|A(\delta)\times B(\delta)+G(D(\delta))\right|_\delta \lesssim \delta^{-\tau}.
		\end{equation}
		To show \eqref{form-coverofsum}, we first prove that
		\begin{equation}\label{form-covernumber}
			|\psi(A(\delta))+B(\delta)|_\delta \sim \delta^{-\tfrac{2\tau}{3}}.
		\end{equation}
		Since $k\lesssim \delta^{-\tfrac{\tau}{3}}$, we observe that
		\[\psi(I_k)=\left(k^2\delta^{\tfrac{2\tau}{3}}+\delta^2-2k\delta^{\tfrac{\tau}{3}+1}, k^2\delta^{\tfrac{2\tau}{3}}+\delta^2+2k\delta^{\tfrac{\tau}{3}+1}\right)\subset \left(k^2\delta^{\tfrac{2\tau}{3}}-3\delta, k^2\delta^{\tfrac{2\tau}{3}}+3\delta\right),\]
		which means $\psi(A(\delta))\subset B(3\delta)$. Since $B$ is an arithmetic progression, we easily get
		\[\delta^{-\tfrac{2\tau}{3}}\lesssim|\psi(A(\delta))+B(\delta)|_{3\delta} \leq |B(3\delta)+B(3\delta)|_{3\delta}\sim \delta^{-\tfrac{2\tau}{3}},\]
		which implies \eqref{form-covernumber}. Now \eqref{form-coverofsum} follows easily by applying \eqref{form-observation} and \eqref{form-covernumber}:
		\begin{equation*}
			\left|A(\delta)\times B(\delta)+G(D(\delta))\right|_\delta \leq \left|(A(\delta)+A(\delta))\times (\psi(A(\delta))+B(\delta))\right|_\delta \lesssim \delta^{-\tau}.
		\end{equation*}
		
		\textbf{Step 2: constructing suitable measures.} Define measures
		\[\mu:=\frac{\mathcal{L}^2|_{A(\delta)\times B(\delta)}}{\mathcal{L}^2(A(\delta)\times B(\delta))}~~\text{and}~~\sigma:=G_{\text{\#}}\left(\frac{\mathcal{L}^1|_{D(\delta)}}{\mathcal{L}^1(D(\delta))}\right),\]
		where $\mu$ is the normalized Lebesgue measure on $A(\delta)\times B(\delta)$ and $\sigma$ is the pushforward of normalized Lebesgue measure on $D(\delta)$. Since both $A$ and $B$ are arithmetic progressions, we can verify that $\mu$ is a $(\tau, c)$-Frostman measure and $\sigma$ is an $(s,c)$-Frostman measure for some absolute constant $c\geq1$. Indeed, for $\delta^{\tfrac{\tau}{3}}\leq r\leq 1$ and $x\in\mathbb R^2$,
		\[\mu(B(x,r))\sim \delta^{\tau-2}\mathcal{L}^2 (A(\delta)\times B(\delta)\cap B(x,r))\lesssim \delta^{\tau-2} \cdot \frac{r^2}{\delta^\tau}\cdot \delta^2\leq r^\tau.\]
		For $r>1$ and $x\in\mathbb R^2$, we simply have $\mu(B(x,r))\leq 1\leq r^\tau$. For $\delta\leq r\leq\delta^{\tfrac{\tau}{3}}$ and $x\in\mathbb R^2$, $B(x,r)$ intersects $\lesssim1$ cubes in $A(\delta)\times B(\delta)$, thus $\mu(B(x,r))\lesssim \delta^{\tau-2}\cdot \delta^2\leq r^\tau$. For $0<r<\delta$ and $x\in\mathbb R^2$, there holds
		\[\mu(B(x,r))\lesssim \delta^{\tau-2}\cdot r^2 =r^\tau\cdot\left(\frac{r}{\delta}\right)^{2-\tau}\leq r^\tau.\]
		Combining the four cases above, we see that $\mu$ is $(\tau, c)$-Frostman and also $\textup{I}_{t}(\mu)\leq c_1$ for some $c_1=c_1(t,\tau)>0$ since $t<\tau$. The same argument shows that $\sigma$ is $(s,c)$-Frostman. Now if we define $\bar{\mu}:=\tfrac{1}{\sqrt{c_1}}\mu$ and $\bar{\sigma}:=\tfrac{1}{c}\sigma$, then it is easy to verify that $\bar{\mu}$ and $\bar{\sigma}$ satisfy the assumptions in Question \ref{qt} with parameters $(t, s)$. 
		
		We claim that $\textup{I}_{\tau+\epsilon}(\bar{\mu}\ast\bar{\sigma})\lesssim_{\psi, s, t,\tau, \epsilon} 1$ cannot be true for any $\epsilon>0$. Otherwise, $\textup{I}_{\tau+\epsilon}(\bar{\mu}\ast\bar{\sigma})\lesssim_{\psi, s,t,\tau,\epsilon} 1$ implies $|spt(\bar{\mu}\ast\bar{\sigma})|_\delta \gtrsim_{\psi, s, t,\tau, \epsilon} \delta^{-\tau-\epsilon}$ (see \cite[Lemma 1.4]{MR4528124}), which contradicts \eqref{form-coverofsum} if $\delta$ is small enough, noting that $spt(\bar{\mu})\subset A(\delta)\times B(\delta)$ and $spt(\bar{\sigma})\subset G(D(\delta))$. By the arbitrariness of $\epsilon$, we obtain $\mathfrak{f}(s,t)\leq \tau$. Letting $\tau\to t$, we get $\mathfrak{f}(s,t)\leq t$. Based on the fact $\textup{I}_{t}(\bar{\mu}\ast\bar{\sigma})\lesssim 1$, we conclude $\mathfrak{f}(s,t)=t$ when $t\in [3s,s+1)$ and $s\in[0,\tfrac{1}{2}]$.
	\end{ex}
	
	\begin{ex}[Case $\mathbf{t\in[s+1, 2)}$ and $\mathbf{s\in[0, 1)}$]\label{ex-4} 
		Fix $s\leq t-1<1$. Find a compact set $A\subset [0,1]$ such that $\textup{dim}_\text{H}(A)=s$ and $\textup{dim}_\text{H}(A+A)=\textup{dim}_\text{H}(A+A+A)=t-1+\epsilon$, see \cite[Example 4]{MR2742987} for the construction of such set. Then we have
		\[\textup{dim}_\text{H}((A+A)\times \mathbb R)=t+\epsilon, \quad \textup{dim}_\text{H}(\Gamma\cap (A\times\mathbb R))=s,\]
		which implies
		\[\textup{dim}_\text{H}\big((A+A)\times \mathbb R+ \Gamma\cap (A\times\mathbb R)\big)\leq \textup{dim}_\text{H}((A+A+A)\times \mathbb R)=t+\epsilon.\]
		Following the same argument as in Example \ref{ex-2} and recalling $\textup{I}_t(\mu\ast\sigma)\lesssim 1$, we conclude that $\mathfrak{f}(s,t)=t$ when $t\geq s+1$.
	\end{ex}
	
	\subsection{Application}\label{sub-1.2}
	As an application of Theorem \ref{main}, we next show the sharp $L^6$-decay of Fourier transform of Frostman measures on $\Gamma$. 
	\begin{thm}\label{main2}
		Let $s \in[2/3,1]$. Assume that $\sigma$ is an $(s,1)$-Frostman measure on $\Gamma$.
		Then
		\begin{equation}\label{decay}
			\|\widehat{\sigma}\|_{L^6(B(R))}^6\leq C(\psi,s,\epsilon) R^{1-s+\epsilon}, \quad R\geq 1, ~\forall\epsilon\in(0,1).
		\end{equation}
	\end{thm}
	\begin{proof}[Proof of Theorem \ref{main2}]
		By applying \eqref{form-main} with $\mu=\sigma$, we get $\textup{I}_{2s-\tau}(\sigma\ast \sigma)\lesssim_{\psi,s, \tau} 1$ for any $\tau>0$. A second use of Theorem \ref{main} with $\mu=\sigma\ast\sigma$ gives $\textup{I}_{s+1-\tau}(\sigma\ast \sigma\ast \sigma)\lesssim_{\psi,s, \tau} 1$ when $s\geq 2/3+\tau/3$. Letting $\tau\to 0$, we obtain that $\textup{I}_{s+1-\epsilon}(\sigma\ast \sigma\ast \sigma)\lesssim_{\psi, s, \epsilon} 1$ if $s\geq 2/3$. By the second identity in \eqref{Riesz}, we deduce 
		\[\|\widehat{\sigma}\|_{L^6(B(R))}^6=\|\widehat{\sigma\ast \sigma\ast \sigma}\|_{L^2(B(R))}^2\lesssim_{\psi, s, \epsilon} R^{1-s+\epsilon},\quad \forall\epsilon\in(0,1),\]
		as desired.
	\end{proof}
	The sharp exponent for all $s\in[0,1]$ in \eqref{decay} was conjectured to be $2-\min\{3s, 1+s\}$, see \cite[Example 1.8]{MR4528124}. When $s\geq 2/3$, Orponen-Puliatti-Py\"or\"al\"a \cite{orponen2024fourier} also established \eqref{decay} when $\psi(x)=x^2$. For general planar curve with non-zero curvature, see \cite{MR4661079} for the bound "$2-2s-\beta(s)$" with $\beta(s)>0$ a small implicit constant. Recently, Demeter-Wang \cite[Theorem 1.2]{demeterhong2024} improved this to "$2-2s-\tfrac{s}{4}$" when $s\in [0,\frac{1}{2}]$ by using their Szemer\'edi–Trotter type of incidence theorem.
	
	\subsection{Proof ideas}\label{sub-1.3}
	It suffices to show inequalities \eqref{form-main3}. Instead, we prove the following stronger $L^2$-inequality which implies \eqref{form-main3}: 
	\begin{equation}\label{form-ineq1}
		\|(\mu\ast\sigma)_\delta\|_{L^2}^2\lesssim_{\psi,s,t,\epsilon} \delta^{\zeta(s,t)-2-\epsilon}, \quad \forall \epsilon\in(0,1),
	\end{equation}
	where $(\mu\ast\sigma)_\delta=\mu\ast\sigma\ast\eta_\delta$ with $\eta_\delta(x)=\delta^{-2}\eta(\frac{x}{\delta})$ the rescaled mollifier and
	\begin{equation}\label{form-parameter}
		\zeta(s,t)=
		\begin{cases}
			s+t, & \text{when} ~~ t\in(0,s],~s\in (0, 1],\\
			2s+t-1, & \text{when}~~ t\in (0, 2-s],~s\in [0, 1].\\
		\end{cases}
	\end{equation}
	\begin{remark}\label{rmk-reduction}
		Here we give some comments on why \eqref{form-ineq1} implies \eqref{form-main3}. For $\zeta(s,t)=s+t$ when $t\in(0,s]$, there is nothing to say about the implication. For $\zeta(s,t)=2s+t-1$ when $s+t\leq 2$, we claim that this implies \eqref{form-ineq1} with $2s+t-1$ replaced by $s+1$ when $s+t\geq 2$:
		\begin{equation}\label{form-uppercase}
			\|(\mu\ast\sigma)_\delta\|_{L^2}^2\lesssim_{\psi,s,t,\epsilon} \delta^{s-1-\epsilon}, \quad \forall \epsilon\in(0,1).
		\end{equation}
		Indeed, if $s+t\geq2$, we denote $\bar{t}:=2-s$, then $\textup{I}_{\bar{t}}(\mu)\lesssim \textup{I}_{t}(\mu)<+\infty$ and $s+\bar{t}\leq 2$. Using result for the case $s+t\leq 2$, we readily see $\zeta(s,\bar{t})=2s+\bar{t}-1=s+1$ which proves \eqref{form-uppercase}. Then it is clear that \eqref{form-uppercase} implies \eqref{form-main3} when $t\in[2-s,s+1]$.
	\end{remark}
	
	We will further reduce \eqref{form-ineq1} to its $\delta$-discretised version, then proving the discretised version will mainly depend on incidence estimates between $\delta$-cubes and curved $\delta$-tubes, see Section \ref{sec2} for the terminology and the reduction. In precise, we need to extend Fu-Ren's estimate \cite[Theorem 5.2]{MR4751206} to curvilinear case. This is easy when $\psi(x)=x^2$ because parabola can be transformed into line by the map
	\[\Psi:\mathbb R^2 \to \mathbb R^2, \quad \Psi(x,y):=(x, x^2-y).\]
	However, if $\psi$ is an arbitrary $C^2$-curve with $\psi''\neq 0$, we do not have a general map to transform $\Gamma$ into line, which is the main challenge in our problem. It turns out the case $t\in(0,s]$ is easier than the case $t\in (0,2-s]$ since an elementary incidence estimate can be obtained by analyzing the geometry of the intersection of two curved $\delta$-tubes, see Proposition \ref{pro-incidence2}. As for $t\in (0,2-s]$, we will use the method in \cite[Section 2]{MR4722034}. In fact, the following $\delta$-incidence theorem of measures will be established, which is the main tool to build incidence estimates between $\delta$-cubes and curved $\delta$-tubes. The notations will be properly defined in Subsection \ref{sub-2.1}. 
	\begin{thm}\label{wei}
		Let $\mu$ and $\nu$ both be finite Borel measures with compact support on $B(1)$. Then for any $\delta\in(0,1)$, there holds
		\begin{equation}\label{form-wei}
			\mathcal{I}_\delta(\mu,\nu)\lesssim_{\psi,t} \delta\sqrt{\textup{I}_{3-t}(\mu) \textup{I}_{t}(\nu)}, \quad t\in (1,2).
		\end{equation}
	\end{thm}
	Theorem \ref{wei} will follow from the Sobolev estimates (Lemma \ref{Sobolev1}) of an operator $\mathfrak{R}$ (Definition \ref{def:R}). We mention that this is also why the estimate \eqref{form-wei} does not contain small error $\delta^{-\epsilon}$ as we will see from the proofs. Let us note the idea of using operators to study incidences has also been applied in \cite{MR2836125,MR2538607}.
	
	We do not know what the precise value of "$\mathfrak{f}(s,t)$" should be for the remaining cases of Question \ref{qt}. Our guess is that this problem has the similar numerology to Furstenberg set problem, see \cite{OS23, OSABC, KevinHong} for the solution of Furstenberg set problem in the plane. This means the following conjecture may be plausible:
	\begin{conjecture}\label{cj}
		For $s\in(0,\frac{1}{2}]$ and $t\in [s,3s]$, or $s\in [\frac{1}{2},1]$ and $t\in [s, 2-s]$, there holds 
		\begin{equation}\label{cj1}
			\mathfrak{f}(s,t)=\frac{3s+t}{2}.
		\end{equation}
	\end{conjecture}
	We remind that all the other cases have been included in \eqref{form-main}. From the proof of Theorem \ref{main2}, we note that a full resolution of Conjecture \ref{cj} would improve Demeter-Wang's bound to "$2-2s-\frac{s}{2}$" for all $s\in [0, \frac{2}{3}]$. 
	
	\subsection{Outline of the paper} Section \ref{sec2} is the preliminary of this paper, in which we will reduce Theorem \ref{main} to some incidence estimates. Section \ref{sec3} contains the proof of Theorem \ref{main} for $t\in [2-s, s+1]$. We will mainly prove Theorem \ref{wei} and use it to build a weighted incidence estimate in Corollary \ref{incidence11}. Once this has been built, Fu-Ren's estimate \cite[Theorem 5.2]{MR4751206} can be extended to curvilinear case whose proof we put in Appendix \ref{apend A}. Finally, The proof of Theorem \ref{main} for the case $t\in (0,s]$ will be contained in Section \ref{sec4}, where an improved incidence estimate can be obtained by simple geometric argument. 
	
	\subsection*{Acknowledgement} The author would like to thank Tuomas Orponen for many useful comments and his constant support.

	\section{Preliminary}\label{sec2}
	We will give all the notations and definitions needed in this paper. Also, we will carry on a reduction process so that Theorem \ref{main} is reduced to some incidence estimates.
	
	\subsection{Notations and definitions}\label{sub-2.1} 
	We adopt the notations $\lesssim$, $\gtrsim$, $\sim$. For example, $A\lesssim B$ means $A\leq CB$ for some constant $C>0$, while $A\lesssim_rB$ stands for $A\leq C(r)B$ for a positive function $C(r)$. 
	
	For $\delta \in 2^{-\N}$, dyadic $\delta$-cubes in $\R^2$ are denoted $\mathcal{D}_{\delta}(\R^2)$. For $P \subset \R^2$, we write $\mathcal{D}_{\delta}(P) := \{p \in \mathcal{D}_{\delta}(\R^{d}) : P \cap p \neq \emptyset\}$. In particular, we write $\mathcal{D}_{\delta}([0,1]^2) :=\mathcal{D}_\delta$. 
	
	\begin{definition}[Curved tubes]\label{Curved tubes}
		Let $\delta\in(0,1]$ and $P\subset\mathbb R^2$. For each $q\in P$, let $q+\Gamma$ be the translation of $\Gamma$ by $q$. If $q\in\mathcal{D}_\delta$,  $q+\Gamma$ means the translation of $\Gamma$ by the center of $q$. 
		
		A curved $\delta$-tube is the closed $\delta$-neighborhood of $q+\Gamma$ for some $q\in P$. For simplicity, we let $\Gamma_q:=q+\Gamma$ and write $\Gamma_q(\delta)$ as the corresponding curved $\delta$-tube. We call $P$ the parameter set of $\mathcal{T}:=\{\Gamma_q(\delta): q\in P\}$.
	\end{definition}
	By definition \ref{Curved tubes}, a set of curved $\delta$-tubes $\mathcal{T}$ is uniquely determined by its parameter set $P$. In this paper, we will mainly consider curved $\delta$-tubes with parameter set $\mathcal{P}\subset\mathcal{D}_\delta$. The next definition of Katz-Tao condition was originally introduced by Katz and Tao \cite{MR1856956}.
	
	\begin{definition}[Katz-Tao $(\delta,s,C)$-set]\label{def:katzTaoSet}
		Let $P\subset\mathbb R^2$ be a bounded set. Let $\delta\in (0,1]$, $0\leq s\leq 2$ and $C\geq 1$. We say that $P$ is a \emph{Katz-Tao $(\delta,s,C)$-set} if
		\begin{equation}\label{def 1}
			|P\cap B(x,r)|_\delta\leq C\left(\tfrac{r}{\delta}\right)^s, \qquad x \in \R^{2}, \, \delta \leq r \leq 1. 
		\end{equation}
		The notation $|\cdot|_{\delta}$ refers to the $\delta$-covering number. If $\mathcal{P} \subset \mathcal{D}_{\delta}(\R^{2})$, we say that $\mathcal{P}$ is a Katz-Tao $(\delta,s,C)$-set if $P := \cup \mathcal{P}$ satisfies \eqref{def 1}. A set of curved $\delta$-tubes $\mathcal{T}$ is said to be a Katz-Tao $(\delta,s,C)$-set if its parameter set is.
	\end{definition}
	
	The following weighted Katz-Tao condition is only defined for dyadic elements. This condition enables us to prove a strong weighted incidence estimate in Corollary \ref{incidence11}.
	\begin{definition}[Weighted Katz-Tao condition]\label{def:weightedkatzTaoSet}
		Let $\delta\in (0,1]$, $s\in[0,2]$ and $C\geq 1$. Let $\mathcal{P}\subset\mathcal{D}_\delta$. For each $q\in\mathcal{P}$, there is a positive integer weight $w(q)$ associated with it. We say that $\mathcal{P}$ is a \emph{weighted Katz-Tao $(\delta,s,C)$-set} if for any $Q\in\mathcal{D}_r$ with $r\in[\delta,1]$ there holds
		\begin{equation}\label{form-KT}
			\sum_{q\in \mathcal{P}\cap Q} w(q)\leq C \left(\tfrac{r}{\delta}\right)^s.
		\end{equation}
		Let $\mathcal{T}:=\{\Gamma_q(\delta): q\in \mathcal{P}\subset \mathcal{D}_\delta\}$ be a set of curved $\delta$-tubes, then we say $\mathcal{T}$ is a weighted Katz-Tao $(\delta,s,C)$-set if its parameter set $\mathcal{P}$ is.
	\end{definition}
	\begin{remark}
		When $w\equiv1$, a weighted Katz-Tao set $\mathcal{P}\subset \mathcal{D}_\delta$ becomes a Katz-Tao set in the sense of Definition \ref{def:katzTaoSet}. Since both notions are needed for us, we will always be careful and explicit in either including the word "weighted", or omitting it. Also, if $\mathcal{P}$ is a weighted Katz-Tao $(\delta,s,C)$-set, we can deduce $w(q)\leq C$ for any $q\in\mathcal{P}$ by taking $r=\delta$ in \eqref{form-KT}.
		
		Moreover, when we say $q\in \mathcal{P}$ has multiplicity $w(q)$, we simply mean that when counting incidences $q$ appears $w(q)$ times, see the following definition.
	\end{remark}
	\begin{definition}[Weighted incidences]
		Let $\mathcal{F}\subset \mathcal{D}_\delta$ be a set of dyadic $\delta$-cubes and let $\mathcal{T}=\{\Gamma_q(\delta): q\in \mathcal{P}\subset\mathcal{D}_\delta\}$ be a set of curved $\delta$-tubes. If weight functions $w_1$ and $w_2$ are associated with $\mathcal{P}$ and $\mathcal{F}$ respectively, then we define the weighted incidences between $\mathcal{F}$ and $\mathcal{T}$ by
		\begin{equation*}
			\mathcal{I}_w(\mathcal{F}, \mathcal{T}):=\sum_{q\in\mathcal{P}} \sum_{p\in\mathcal{F}}w_1(q)w_2(p) \boldsymbol{1}_{\{p\cap \Gamma_q(\delta)\neq \emptyset\}}.
		\end{equation*}
	\end{definition}
	
	To get a nice upper bound for $\mathcal{I}_w(\mathcal{F}, \mathcal{T})$, we will need to first study the $\delta$-incidences between two finite Borel measures.
	\begin{definition}[$\delta$-incidences]\label{def:deltaincidence}
		Let $\mu$ and $\nu$ both be finite Borel measures on $\mathbb R^2$. For any $\delta\in(0,1)$ and $q\in \mathbb R^2$, the $\delta$-incidences are defined as
		\begin{equation*}
			\mathcal{I}_\delta(\mu, \nu):=\mu\times \nu\big(\{(p, q)\in \mathbb R^2\times \mathbb R^2: p\in \Gamma_q(\delta)\}\big).
		\end{equation*}
	\end{definition}
	
	\subsection{Reduction to incidence estimates}\label{sub-2.2}
	By Remark \ref{rmk-reduction}, to show Theorem \ref{main} it suffices to show the following $L^2$-smoothing estimate.
	\begin{thm}\label{L2-estimates}
		Let $s\in[0,1]$ and $t\in (0,2)$. Let $\mu$ be a Borel measure on $B(1)$ such that $\textup{I}_t(\mu)<+\infty$. Assume $C_\sigma\geq 1$ and let $\sigma$ be an $(s, C_\sigma)$-Frostman measure on $\Gamma$. Then
		\begin{equation}\label{form-L2}
			\|(\mu\ast\sigma)_\delta\|^2_{L^2(\mathbb R^2)}\lesssim_{\psi, s,t,\epsilon} \max\{\textup{I}_t(\mu), 1\}\cdot C_\sigma ^2\cdot\delta^{\zeta(s,t)-2-\epsilon}, \quad \epsilon\in(0,1),
		\end{equation}
		where 
		\begin{equation}\label{form-zeta}
			\zeta(s,t)=
			\begin{cases}
				s+t, & \text{when} ~~ t\in(0,s],~s\in (0, 1],\\
				2s+t-1, & \text{when}~~ t\in (0, 2-s],~s\in [0, 1].\\
			\end{cases}
		\end{equation}
		In particular, $\textup{I}_{\zeta(s,t)-\epsilon}(\mu\ast\sigma)\lesssim_{\psi, s,t,\epsilon} \max\{\textup{I}_t(\mu), 1\}\cdot C_\sigma ^2$~~for any $\epsilon\in(0,1)$. 
	\end{thm}
	\begin{remark}
		(i) If $t\in [s,2-s]$, by \eqref{form-L2} we have $\mathfrak{f}(s,t)\geq 2s+t-1$, which is all we can say about the intermediate case. This is far from the conjectured bound \eqref{cj1} noting that $(t+2s-1)<(3s+t)/2$ when $t<2-s$.
		
		(ii) By Plancherel identity, Theorem \ref{L2-estimates} implies the $L^4$-estimate of $\sigma$, that is, 
		\[\|(\sigma\ast\sigma)_\delta\|^2_{L^2(\mathbb R^2)}\lesssim\delta^{2s-2-\epsilon}\Longrightarrow \|\hat{\sigma}\|_{L^4(B(R))}^4\lesssim R^{2-2s+\epsilon}.\]
		This was established in \cite[Section 3]{MR4528124} when $\psi(x)=x^2$, in which a Fourier analytic proof was provided. Our proof will be simple and depends on an incidence estimate. When $s=1$, the $L^4$-estimate of $\hat{\sigma}$ also follows from the sharp restriction estimate, see for example \cite[Theorem 1.14]{MR3971577}.
	\end{remark}
	
	Before the reduction process, we take the definition of $\delta$-measures from \cite[Section 6]{orponen2024fourier}.
	\begin{definition}[$\delta$-measure]\label{def:measure}
		A collection of non-negative weights $\boldsymbol{\mu}=\{\mu(p)\}_{p\in\mathcal{D}_\delta}$ with $\|\boldsymbol{\mu}\|:=\sum_p\mu(p)\leq 1$ is called a $\delta$-measure. Let $C_\mu\geq1$. We say that a $\delta$-measure $\boldsymbol{\mu}$ is a $(\delta,s,C_\mu)$-measure if $\mu(q)\leq C_\mu r^s$ for all $q\in\mathcal{D}_r$ with $r\in[\delta,1]$. Moreover, the $s$-energy of $\boldsymbol{\mu}$ is defined as
		\[\textup{I}_s(\boldsymbol{\mu}):=1+\sum_{p\neq q}\frac{\mu(p)\mu(q)}{\text{dist}(p,q)},\]
		where $\text{dist}(p,q)$ is the distance between the midpoints of $p$ and $q$.
	\end{definition}
	
	Now we begin to make the reductions. Namely, Theorem \ref{L2-estimates} can be first reduced to Proposition \ref{pro-dis}. This reduction is standard and we will omit the proof as one can check \cite[Section 6]{orponen2024fourier} for the details when $\psi(x)=x^2$ (note the proof generalizes easily to our case). Here we list it for the reader's convenience.
	\begin{proposition}\label{pro-dis}
		Let $s\in[0,1]$ and $t\in (0,2)$. Assume $\mathbf{C}_\mu, \mathbf{C}_\sigma\geq 1$. Let $\boldsymbol{{\mu}}$ be a $(\delta, t, \mathbf{C}_\mu)$-measure. For each $q\in \mathcal{D}_\delta$, let $\boldsymbol{\sigma}_q$ be a $(\delta, s, \mathbf{C}_\sigma)$-measure supported on $\{p\in\mathcal{D}_\delta: p\cap \Gamma_q(\delta)\neq \emptyset\}$. Let $\zeta(s,t)$ be the same as \eqref{form-zeta}. Then for any $\epsilon\in (0,1)$, there holds
		\begin{equation}\label{form-dis}
			\int \left(\sum_{q\in \mathcal{D}_\delta}\mu(q)\sum_{p\in\mathcal{D}_\delta}\boldsymbol{\sigma}_q(p)\cdot(\delta^{-2}\mathbf{1}_p)\right)^2 dx \leq_{\psi, s, t,\epsilon} (\mathbf{C}_\mu\mathbf{C}_\sigma^2\|\boldsymbol{\mu}\|)\cdot \delta^{\zeta(s,t)-2-\epsilon}+1.
		\end{equation}
	\end{proposition}
	
	Then Proposition \ref{pro-dis} can be reduced to Theorem \ref{id-thm}. The proof of this reduction will also be skipped since it is almost the same as the proof of \cite[Proposition 6.7]{orponen2024fourier} if we have the following incidence estimate. 
	\begin{thm}\label{id-thm}
		Let $s\in[0,1]$ and $t\in [0,2]$ and let $A, B\geq 1$. Assume $\mathcal{P}\subset \mathcal{D}_\delta$ is a Katz-Tao $(\delta, t, A)$-set, and for each $q\in \mathcal{P}$ there exists a Katz-Tao $(\delta, s, B)$-set $\mathcal{F}(q)\subset \{p\in \mathcal{D}_\delta: p\cap \Gamma_q(\delta)\neq \emptyset\}$. Write $\mathcal{F}:=\bigcup_{q\in \mathcal{P}} \mathcal{F}(q)$. Then
		\begin{equation}\label{id0}
			\sum_{q\in \mathcal{P}}|\mathcal{F}(q)|\lesssim_{\psi,s,t,\epsilon} \sqrt{\delta^{-\gamma(s,t)-\epsilon} AB |\mathcal{F}||\mathcal{P}|},
		\end{equation}
		where
		\begin{equation}\label{form-gamma}
			\gamma(s,t)=
			\begin{cases}
				s, & \text{when} ~~ t\in[0,s],~s\in [0, 1],\\
				1, & \text{when}~~ t\in [0, 2-s],~s\in [0, 1].
			\end{cases}
		\end{equation}
	\end{thm}
	
	When $s+t\leq 2$, the incidence estimate \eqref{id0} was first established by Fu-Ren \cite[Theorem 5.2]{MR4751206} in linear setting, then it was also proved when $\psi(x)=x^2$ in \cite[Theorem 6.12]{orponen2024fourier}. In this paper, we can actually prove a $\delta^{-\epsilon}$-free version of \eqref{id0} when $s\in[0,1)$ and $t\in [0,2)$ with $s+t<2$, see Theorem \ref{thm-incidence1}. As a consequence, we have finished the reduction process and the remaining task of this paper is to prove Theorem \ref{id-thm} in both cases, see Section \ref{sec3} and Section \ref{sec4}.

	\section{Proof of Theorem \ref{main} for $\mathbf{t\in[2-s,s+1]}$ }\label{sec3}
	In this section, we will prove Theorem \ref{main} when $t\in[2-s,s+1]$ and $s\in[\frac{1}{2},1]$. Thanks to the reduction in Subsection \ref{sub-2.2}, it suffices to establish the following incidence theorem.
	\begin{thm}\label{thm-incidence1}
		Let $s\in[0,1)$ and $t\in [0,2)$ such that $s+t<2$ and let $A, B\geq 1$. Assume $\mathcal{P}\subset \mathcal{D}_\delta$ is a Katz-Tao $(\delta, t, A)$-set, and for each $q\in \mathcal{P}$ there exists a Katz-Tao $(\delta, s, B)$-set $\mathcal{F}(q)\subset \{p\in \mathcal{D}_\delta: p\cap \Gamma_q(\delta)\neq \emptyset\}$. Write $\mathcal{F}:=\bigcup_{q\in \mathcal{P}} \mathcal{F}(q)$. Then
		\begin{equation}\label{form-incidence1}
			\sum_{q\in \mathcal{P}}|\mathcal{F}(q)|\lesssim_{\psi,s,t} \sqrt{\delta^{-1}AB |\mathcal{F}||\mathcal{P}|}.
		\end{equation}
	\end{thm}
	The incidence estimate \eqref{form-incidence1} with a $\delta^{-\epsilon}$-error is also true for any $s\in[0,1]$ and $t\in [0,2]$ with $s+t\leq 2$, but here we only state a $\delta^{-\epsilon}$-free version. This is due to the method we will use. As mentioned in Subsection \ref{sub-1.3}, a $\delta$-incidence theorem of measures will be first established which we restate as below. Once this has been obtained, the proof of Theorem \ref{thm-incidence1} will be a modification of \cite[Theorem 5.2]{MR4751206}, which we put in Appendix \ref{apend A} for completeness. In the sequel, our main task is to prove Theorem \ref{thm-incidence-measure}. The reader may recall Definition \ref{def:deltaincidence} for the notion of $\delta$-incidences.
	\begin{thm}\label{thm-incidence-measure}
		Let $\mu$ and $\nu$ both be finite Borel measures with compact support on $B(1)$. Then for any $\delta\in(0,1)$, there holds
		\begin{equation}\label{form-incidence-measue}
			\mathcal{I}_\delta(\mu,\nu)\lesssim_{\psi,t} \delta\sqrt{\textup{I}_{3-t}(\mu) \textup{I}_{t}(\nu)}, \quad t\in (1,2).
		\end{equation}
	\end{thm}
	
	\subsection{Sobolev estimates}\label{sub-3.1}
	Define a measure $\mu$ by restricting $\mathcal{H}^1$ to $\Gamma$, then
	\begin{equation*}
		\hat{\mu}(\xi)=\int e^{-2\pi i w\cdot\xi}d\mu(w)=\int_{-1}^1 e^{-2\pi i (x, \psi(x))\cdot \xi} \sqrt{1+\psi'(x)^2} dx,\quad \xi\in\mathbb R^2.
	\end{equation*}
	
	We first use the method of stationary phase to prove the decay estimate for $\hat{\mu}$, see \cite[Chapter 14]{MR3617376} or \cite[Chapter 6]{MR2003254} for some related discussions. 
	
	\begin{lemma}\label{Fourier decay}
		Under assumption \eqref{curvature assumption}, $|\hat{\mu}(\xi)|\leq C(\psi) |\xi|^{-1/2}$ for any $\xi\in \mathbb R^2$.
	\end{lemma}
	\begin{proof}
		Write $\xi=\lambda e$ for some $\lambda>0$ and $e\in S^1$, and $\psi_e(x)=-2\pi (xe_1+\psi(x) e_2)$. We may assume $\lambda>1$ since the result is trivial when $\lambda\in(0,1]$. It then suffices to show
		\begin{equation}\label{stationary phase}
			|\hat{\mu}(\xi)|=|\hat{\mu}(\lambda e)|=\left |\int_{-1}^1 e^{i\lambda \psi_e(x)} \sqrt{1+\psi'(x)^2} dx\right |\lesssim_{\psi} \lambda^{-1/2}.
		\end{equation}
		By simple calculation, we have 
		\[\psi_e'(x)=-2 \pi(e_1+\psi'(x) e_2), \quad \psi_e''(x)=-2 \pi \psi''(x) e_2.\]
		Write $c_1:=\max_{x\in[-1,1]} |\psi'(x)|>0$. We divide the proof into two cases.
		
		\textbf{Case 1: $\mathbf{|e_2|<\min\{\frac{1}{2}, \frac{1}{2c_1}\}}$}. Then we easily see
		\[|\psi_e'(x)|\geq 2\pi (|e_1|-|\psi'(x) e_2|)\geq 2\pi (\frac{\sqrt{3}-1}{2})> 1, \quad x\in[-1,1].\]
		It follows from integrating by parts
		\[\begin{split}
			|\hat{\mu}(\xi)|&=\left|\int_{-1}^{1} \frac{\sqrt{1+\psi'(x)^2}}{i\lambda\psi_e'(x)} de^{i\lambda \psi_e(x)}\right|\\
			&\leq \frac{C(\psi)}{\lambda}+C(\psi)\lambda^{-1} \int_{-1}^{1} \left|d \frac{1}{\psi_e'(x)} \right|\leq \frac{C(\psi)}{\lambda}, \quad \text{as desired.}
		\end{split}\] 
		
		\textbf{Case 2: $\mathbf{|e_2|\geq\min\{\frac{1}{2}, \frac{1}{2c_1}\}}$}. Then by \eqref{curvature assumption} we can deduce
		\[\min_{x\in [-1,1]}|\psi_e''(x)|=2\pi |e_2| \cdot \min_{x\in [-1,1]}|\psi''(x)| \geq c_2=c_2(\psi)>0.\]
		Without loss of generality, we assume $\psi_e''>0$ on $[-1,1]$. Assume that $\psi_e$ attains its minimum on $[-1,1]$ at $x_0\in[-1,1]$, then either $\psi_e'(x_0)=0$ or $x_0=\pm 1$. We first consider the case $\psi_e'(x_0)=0$. For any $\delta>0$, we have $\psi_e'(x)>c_2\delta$ for all $x\in [-1,1]\backslash [x_0-\delta, x_0+\delta]$, which implies
		\[\begin{split}
			\left|\int_{-1}^{x_0-\delta}e^{i\lambda \psi_e(x)} \sqrt{1+\psi'(x)^2} dx\right|&=\left|\int_{-1}^{x_0-\delta} \frac{\sqrt{1+\psi'(x)^2}}{i\lambda\psi_e'(x)} de^{i\lambda \psi_e(x)}\right|\\
			&\leq \frac{C(\psi)}{\lambda \cdot c_2\delta}+C(\psi)\lambda^{-1} \int_{-1}^{x_0-\delta} \left|d \frac{1}{\psi_e'(x)} \right|\leq \frac{C(\psi)}{\lambda \cdot \delta}, 
		\end{split}\] 
		where the constant $C(\psi)>0$ may change from line to line. A similar argument also shows that
		\[\left|\int_{x_0+\delta}^{1}e^{i\lambda \psi_e(x)} \sqrt{1+\psi'(x)^2} dx\right|\leq \frac{C(\psi)}{\lambda \cdot \delta}. \]
		Also, by continuity we have
		\[\left|\int_{x_0-\delta}^{x_0+\delta}e^{i\lambda \psi_e(x)} \sqrt{1+\psi'(x)^2} dx\right|\leq C(\psi)\cdot 2\delta.\]
		Combining the three inequalities above and taking $\delta=\lambda^{-1/2}$, we get
		\[|\hat{\mu}(\xi)|\leq C(\psi) \lambda^{-1/2}.\]
		Here we only considered the case $-1\leq x_0-\delta, x_0+\delta \leq 1$ since other cases can be calculated similarly. Regarding $x_0=\pm 1$, it also follows from the same argument. Combining the two cases above, we thus finish the whole proof of \eqref{stationary phase}.
	\end{proof}
	
	To proceed, define a convolution-type operator $\mathfrak{R}$ by
	\begin{equation}\label{def:R}
		\mathfrak{R}f:=f\ast \mu,  \quad  \forall f\in C^\infty_c(\mathbb R^2),
	\end{equation}
	then
	\begin{equation}\label{operator}
		\mathfrak{R}f(x,y)=\int_{\{y-t=\psi(x-s)\}} f(s,t) d\mathcal{H}^1(s,t).
	\end{equation}
	
	\begin{remark}\label{rmk-1}
		Our operator is similar to the $X$-ray transform which maps $g\in C^\infty_c(\mathbb R^2)$ to a function defined on the set of all lines in $\mathbb R^2$:
		\[(Xg)(\ell):=\int_{\ell} g=\int_{\pi_\theta^{-1}\{r\}}g(z)d \mathcal{H}^1(z),\quad (\theta,r)\in [0,1]\times \mathbb R,\]
		where $\pi_\theta(z)=z\cdot (\cos{2\pi\theta}, \sin{2\pi\theta})$ for $z\in\mathbb R^2$. If we replace $\psi$ by $\widetilde{\psi}(x)=-\psi(-x)$ in our definition, then the right side of \eqref{operator} becomes
		\begin{equation}\label{form-2.1}
			\int_{(x,y)+\Gamma} f(s,t) d\mathcal{H}^1(s,t)=:\widetilde{\mathfrak{R}}f(x,y).
		\end{equation}
		Hence the difference is that translations of $\Gamma$ are parametrized by $(x,y)\in\mathbb R^2$ while the lines in $X$-ray transform are parametrized by $(\theta,r)\in [0,1]\times \mathbb R$. Note that the results in this section hold for all $\widetilde{\psi}\in C^2(\mathbb R^2)$ with non-zero curvature.
	\end{remark}
	
	For $s>-1$, let $\dot{H}^s(\mathbb R^2)$ be the homogeneous Sobolev space. Recall that the norm in $\dot{H}^s(\mathbb R^2)$ is given by
	\[\|f\|_{\dot{H}^s}=\left(\int |\hat{f}(\xi)|^2 |\xi|^{2s} d\xi\right)^{1/2}, \quad f\in C_c(\mathbb R^2).\] 
	Next, we will apply Lemma \ref{Fourier decay} to obtain the Sobolev estimates of $\mathfrak{R}$. 
	\begin{lemma}\label{Sobolev1}
		For any $f\in C^\infty_c(\mathbb R^2)$, there exists a constant $C=C(\psi)>0$ such that
		\[\|\mathfrak{R}f\|_{L^2}\leq C \|f\|_{\dot{H}^{-1/2}}\quad \text{and} \quad \|\mathfrak{R}f\|_{\dot{H}^{1}}\leq C \|f\|_{\dot{H}^{1/2}}.\]
	\end{lemma}
	\begin{proof}
		By Lemma \ref{Fourier decay}, we easily deduce
		\begin{equation*}
			\begin{split}
				\|\mathfrak{R}f\|_{L^2}^2&=\int |\widehat{f\ast \mu}(\xi)|^2 d\xi=\int |\hat{f}(\xi)|^2|\hat{\mu}(\xi)|^2 d\xi\\
				&\lesssim_\psi\int |\hat{f}(\xi)|^2|\xi|^{-1} d\xi=\|f\|_{\dot{H}^{-1/2}}^2,
			\end{split}
		\end{equation*}
		and
		\begin{equation*}
			\begin{split}
				\|\mathfrak{R}f\|_{\dot{H}^{1}}^2&=\int |\widehat{f\ast \mu}(\xi)|^2 |\xi|^2d\xi=\int |\hat{f}(\xi)|^2|\hat{\mu}(\xi)|^2 |\xi|^2 d\xi\\
				&\lesssim_\psi\int |\hat{f}(\xi)|^2|\xi|^{1} d\xi=\|f\|_{\dot{H}^{1/2}}^2.
			\end{split}
		\end{equation*}
	\end{proof}
	The following is a corollary of Lemma \ref{Sobolev1} by standard interpolation argument.
	\begin{cor}\label{Sobolev2}
		There exists a constant $C=C(\psi)>0$ such that
		\[\|\mathfrak{R}f\|_{\dot{H}^{s+1/2}}\leq C \|f\|_{\dot{H}^{s}}, \quad f\in C^\infty_c(\mathbb R^2),~s\in[-1/2, 1/2].\]
	\end{cor}
	\begin{proof}
		Take $\omega_0(\xi)=|\xi|^{-1}, \omega_1(\xi)=|\xi|$ and $v_0(x)=1, v_1(x)=|x|^2$. Let $\mathfrak{F}$ be the Fourier transform operator. By Lemma \ref{Sobolev1}, the operator $(\mathfrak{F}\circ\mathfrak{R})$ extends to a bounded operator $L^2(\mathbb R^2, \omega_0 d\mathcal{L}^2)\to L^2(\mathbb R^2, v_0 d\mathcal{L}^2)$ and $L^2(\mathbb R^2, \omega_1 d\mathcal{L}^2)\to L^2(\mathbb R^2, v_1 d\mathcal{L}^2)$. The next step is to apply Stein-Weiss $L^p$-interpolation (see for example \cite[Theorem 5.4.1]{MR482275}), which gives
		\[\|(\mathfrak{F}\circ\mathfrak{R})(f)\|_{L^2(v_\theta)}\leq C(\psi)\|f\|_{L^2(\omega_\theta)},\]
		where $\omega_\theta=\omega_0^{1-\theta}\omega_1^{\theta}$ and $v_\theta=v_0^{1-\theta}v_1^{\theta}$. After rewriting the inequality and letting $s=\theta-1/2$, we finally get
		\[\|\mathfrak{R}f\|_{\dot{H}^{s+1/2}}\lesssim_\psi \|f\|_{\dot{H}^{s}}.\]
	\end{proof}

	\subsection{Estimating incidences}\label{sub-3.2}
	In this subsection, the Sobolev estimates developed above are applied to establish Theorem \ref{thm-incidence-measure}. The following lemma builds a connection between the $\delta$-incidences and the operator $\mathfrak{R}$ defined in \eqref{def:R}.
	\begin{lemma}\label{connecting inequality}
		Let $q=(x_q, y_q)\in B(1)$ and $\delta\in (0, 1)$, then there exists a constant $c=c(\psi)>0$ such that
		\[\mathbf{1}_{\Gamma_q(\delta)}(p)\leq \delta^{-1}\int_{\mathbb R} \mathbf{1}_{B(q, c\delta)}(x, y_p-\psi(x_p-x)) dx, \quad \forall p=(x_p,y_p)\in B(1).\]
	\end{lemma}
	\begin{proof}
		Fix $p\in B(1)\cap \Gamma_q(\delta)$. It suffices to show $(x, y_p-\psi(x_p-x))\in B(q, c\delta)$ for some $c(\psi)>0$ whenever $|x-x_q|\leq \delta$. Since $p\in \Gamma_q(\delta)$, $dist(p, \Gamma_q)\leq \delta$. Then we can find some $(x_0, y_0)\in \Gamma_q$ such that $|p-(x_0, y_0)|\leq \delta$. Note $y_0=y_q+\psi(x_0-x_q)$, which implies
		\[|(x_p-x_0, y_p-y_q-\psi(x_0-x_q))|\leq \delta.\]
		By applying the triangle inequality and noting that $\psi$ is Lipschitz on bounded intervals, we infer
		\begin{equation*}
			\begin{split}
				|y_p-y_q-\psi(x_p-x)|&\leq |y_p-y_q-\psi(x_0-x_q)|+|\psi(x_0-x_q)-\psi(x_p-x)|\\
				&\leq \delta +C(\psi)|(x-x_q)+(x_0-x_p)|\leq (2C(\psi)+1)\delta.
			\end{split}
		\end{equation*}
		By choosing $c>0$ properly and recalling $|x-x_q|\leq \delta$, we get $(x, y_p-\psi(x_p-x))\in B(q, c\delta)$ and complete the proof.
	\end{proof}
	
	We are now ready to establish the $\delta$-incidence theorem.
	
	\begin{proof}[Proof of Theorem \ref{thm-incidence-measure}]
		We assume $\mu \in C_c^\infty(\mathbb R^2)$. For general case, consider the smooth approximation of $\mu$. Let $\eta\in C_c^\infty(\mathbb R^2)$ satisfy $\mathbf{1}_{B(0, 1/2)}\leq \eta \leq \mathbf{1}_{B(0, 1)}$ and $\int \eta \sim 1$. Write $\eta_\delta(q)=\delta^{-2}\eta(q/\delta)$. For $\delta>0$, define $\nu_\delta:=\nu\ast \eta_{2c\delta}\in C_c^\infty(\mathbb R^2)$ where $c>0$ is the constant in Lemma \ref{connecting inequality}. Note $\nu (B(q, c\delta))\lesssim \delta^2\nu_{\delta}(q)$. Fix $p\in \text{spt}(\mu)\subset B(1)$, then by Lemma \ref{connecting inequality} and Fubini Theorem,
		\begin{equation*}
			\begin{split}
				\int \mathbf{1}_{\Gamma_q(\delta)}(p) d\nu(q)&\leq \int \delta^{-1}\int \mathbf{1}_{B(q, c\delta)}(x, y_p-\psi(x_p-x)) dx d\nu(q)\\
				&=\delta^{-1} \int \nu(B((x, y_p-\psi(x_p-x)), c\delta)) dx\lesssim \delta \int \nu_\delta (x, y_p-\psi(x_p-x)) dx.
			\end{split}
		\end{equation*}
		By definition of $\mathcal{I}_\delta(\mu\times \nu)$, we infer from Fubini Theorem and coarea formula
		\begin{equation*}
			\begin{split}
				\delta^{-1}\mathcal{I}_\delta(\mu\times \nu)&=\delta^{-1} \int \int \mathbf{1}_{\Gamma_q(\delta)}(p) d\nu(q) d\mu(p)\lesssim \int_{\mathbb R^2} \int_{\mathbb R} \nu_\delta (x, y_p-\psi(x_p-x)) dx d\mu(p)\\
				&=\int_{\mathbb R} \int_{\mathbb R^2} \nu_\delta (x, y_p-\psi(x_p-x)) d\mu(p) dx\\
				&=\int_{\mathbb R}\int_{\mathbb R}\nu_{\delta}(x,y)\int_{\{y_p-y=\psi(x_p-x)\}}\frac{\mu(p)}{\sqrt{1+\psi'(x_p-x)^2}}d\mathcal{H}^1(p) dy dx\\
				&\leq \int_{\mathbb R^2}\nu_{\delta}(x,y) \int_{(x,y)+\Gamma}\mu(p) d\mathcal{H}^1(p) dx dy=\int_{\mathbb R^2}\nu_{\delta}(x,y) \widetilde{\mathfrak{R}}\mu(x,y) dx dy.
			\end{split}
		\end{equation*}
		Recall that $\widetilde{\mathfrak{R}}$ was defined in Remark \ref{rmk-1}. Next, by Plancherel identity and Cauchy-Schwarz inequality, we infer
		\begin{equation*}
			\begin{split}
				\delta^{-1}\mathcal{I}_\delta(\mu\times \nu)&\lesssim \int_{\mathbb R^2}\widehat{\nu_{\delta}}(\xi) \widehat{\overset{\sim}{\mathfrak{R}}\mu}(\xi) d\xi\\
				&\lesssim \left(\int |\widehat{\nu}(\xi)|^2 |\xi|^{t-2} d\xi\right)^{1/2}\left(\int |\widehat{\overset{\sim}{\mathfrak{R}}\mu}(\xi)|^2 |\xi|^{2-t} d\xi\right)^{1/2}\\
				&\sim_t \textup{I}_{t}(\nu)^{1/2}\|\widetilde{\mathfrak{R}}\mu\|_{\dot{H}^{(2-t)/2}}. 
			\end{split}
		\end{equation*}
		Since $(2-t)/2=s+1/2$ for $s=(1-t)/2$ and $t\in (1,2)$, it follows from Corollary \ref{Sobolev2} (with $\mathfrak{R}$ replaced by $\widetilde{\mathfrak{R}}$) that
		\[\|\widetilde{\mathfrak{R}}\mu\|_{\dot{H}^{(2-t)/2}}\lesssim_{\psi}\|\mu\|_{\dot{H}^s}\sim_{t} (\textup{I}_{3-t}(\mu))^{1/2}.\]
		Combining the above estimates finishes the proof of \eqref{form-incidence-measue}.
	\end{proof}
	
	Theorem \ref{thm-incidence-measure} yields the following weighted incidence estimate under the weighted Katz-Tao condition we defined in Definition \ref{def:weightedkatzTaoSet}. This weighted incidence estimate is the main tool we will need to prove Theorem \ref{thm-incidence1}. See Appendix \ref{apend A} for the details of how to deduce Theorem \ref{thm-incidence1} from Corollary \ref{incidence11}.
	\begin{cor}\label{incidence11}
		Let $s, t\in [0,2)$ with $s+t< 3$ and let $A, B\geq 1$. There exists $\delta_0=\delta_0(\psi)>0$ such that the following holds for any $\delta\in(0,\delta_0]$. Assume that $\mathcal{T}:=\{\Gamma_q(\delta): q\in \mathcal{P}\subset \mathcal{D}_\delta\}$ is a weighted Katz-Tao $(\delta, t, A)$-set with weight function $w_1$ and $\mathcal{F}\subset \mathcal{D}_\delta$ is a weighted Katz-Tao $(\delta, s, B)$-set with weight function $w_2$. Then
		\begin{equation}\label{incidence22}
			\mathcal{I}_w(\mathcal{F}, \mathcal{T})\lesssim_{\psi,s,t} \sqrt{\delta^{-1}AB \left(\sum_{q\in\mathcal{P}}w_1(q)\right) \left(\sum_{p\in\mathcal{F}}w_2(p)\right)}.
		\end{equation}
	\end{cor}
	\begin{proof}
		Choose $u\in(s,2)$ and $v\in(t,2)$ such that $u+v=3$. Define two measures by
		\[\nu:=\delta^{t-2}\sum_{q\in\mathcal{P}}w_1(q)\boldsymbol{1}_{q},\quad \mu:=\delta^{s-2}\sum_{p\in\mathcal{F}}w_2(p)\boldsymbol{1}_{p}.\]
		Then $spt(\nu)\subset \cup \mathcal{P}$ and $spt(\mu)\subset \cup \mathcal{F}$. Also, we have the useful estimates
		\begin{equation}\label{form-claim1}
			\nu(q)\sim \delta^t w_1(q) ~\text{if} ~q\in \mathcal{P}, \quad \mu(p)\sim \delta^s w_2(p) ~\text{if}~ p\in \mathcal{F}.
		\end{equation}
		By definition of weighted incidences and \eqref{form-claim1}, we can deduce that
		\begin{equation}\label{form-upper}
			\begin{split}
				\mathcal{I}_w(\mathcal{F},\mathcal{T})&\sim \delta^{-s-t}\sum_{q\in\mathcal{P}}\sum_{p\in\mathcal{F}} \nu(q)\mu(p)\boldsymbol{1}_{\{p\cap\Gamma_q(\delta)\neq\emptyset\}}\\
				&=\delta^{-s-t} \mu\times\nu(\{(p,q)\in\mathcal{F}\times\mathcal{P}:p\cap\Gamma_q(\delta)\neq\emptyset\})\\
				&\leq \delta^{-s-t} \mu\times\nu(\{(x,y)\in \mathbb R^4: x\in\Gamma_y(C\delta)\})=\delta^{-s-t} \mathcal{I}_{C\delta}(\mu\times\nu).
			\end{split}
		\end{equation}
		Here we used the fact that there exists $C=C(\psi)>0$ such that
		\[\{(p,q)\in\mathcal{F}\times\mathcal{P}:p\cap\Gamma_q(\delta)\neq\emptyset\}\subset \{(x,y)\in \mathbb R^2\times \mathbb R^2: x\in \Gamma_y(C\delta)\}.\]
		Note that we choose $\delta_0>0$ small such that $C\delta\in(0,1)$ for any $\delta\in (0,\delta_0]$.
		
		To apply Theorem \ref{thm-incidence-measure}, we only need to calculate the Riesz energies of $\nu$ and $\mu$. By using the weighted Katz-Tao condition of $\mathcal{P}$ we compute
		\[\begin{split}
			\textup{I}_v(\nu)&=\sum_{q\in \mathcal{P}} \sum_{\substack{p\in \mathcal{P} \\ dist (p,q)=0 }}\int_{p\times q} \frac{d\nu(x) d\nu(y)}{|x-y|^v}+\sum_{\substack{p,q \in \mathcal{P}\\ dist (p,q)\in [\delta, 1]}} \int_{p\times q} \frac{d\nu(x) d\nu(y)}{|x-y|^v}\\
			&\lesssim \sum_{q\in \mathcal{P}} (\delta^{t-2})^2 w_1(q)A\cdot 9 \int_{3p\times 3p} \frac{dx dy}{|x-y|^v}+\sum_{q\in \mathcal{P}}\sum_{j=1}^{\log\frac{1}{\delta}}\sum_{\substack{dist (p,q) \\ \in [2^{-j}, 2^{-j+1})}} \int_{p\times q} \frac{d\nu(x) d\nu(y)}{|x-y|^v}\\
			&\lesssim A \sum_{q\in \mathcal{P}} w_1(q)(\delta^{t-2})^2\cdot (\delta^{4-v})+\sum_{q\in \mathcal{P}}\sum_{j=1}^{\log\frac{1}{\delta}}\sum_{\substack{dist (p,q) \\ \in [2^{-j}, 2^{-j+1})}} 2^{jv}\nu(p)\nu(q)\\
			&\stackrel{(\ref{form-claim1})}{\lesssim} A\delta^{2t-v}\sum_{q\in \mathcal{P}} w_1(q)+\sum_{q\in \mathcal{P}}w_1(q) \sum_{j=1}^{\log\frac{1}{\delta}}2^{jv}\left(\sum_{p\in\mathcal{P}\cap Q_j}w_1(p)\right)\delta^{2t}\lesssim_t A\delta^{2t-v}\sum_{q\in \mathcal{P}} w_1(q),
		\end{split}\]
		where "$3p$" denotes the cube of side length $3\delta$ with the same center as $p$, in the last inequality $Q_j$ is a dyadic cube with side length $\sim 2^{-j}$ and recall $t-v<0$. By the same calculation, we can obtain $\textup{I}_{u}(\mu)\lesssim_s B\delta^{2s-u}\sum_{p\in \mathcal{F}} w_2(p)$. 
		
		Consequently, we infer by Theorem \ref{thm-incidence-measure} that (note $v>1$ and $u+v=3$)
		\begin{equation*}
			\begin{split}
				\mathcal{I}_w(\mathcal{F}, \mathcal{T})&\overset{\eqref{form-upper}}{\lesssim} \delta^{-s-t}\mathcal{I}_{C\delta}(\mu,\nu)\overset{\eqref{form-incidence-measue}}{\lesssim}_{\psi,v} \delta^{-s-t+1}\sqrt{\textup{I}_v(\nu)\textup{I}_u(\mu)}\\
				&\lesssim_{\psi,s,t} \delta^{-s-t+1}\sqrt{A\delta^{2t-v}\sum_{q\in \mathcal{P}} w_1(q)B\delta^{2s-u}\sum_{p\in \mathcal{F}} w_2(p)}\\
				&=\sqrt{\delta^{-1}AB\sum_{q\in \mathcal{P}} w_1(q)\sum_{p\in \mathcal{F}} w_2(p)},
			\end{split}
		\end{equation*}
		which completes the proof of \eqref{incidence22}.
	\end{proof}
	\begin{remark}\label{rmk-endcase}
		When $s+t=3$ or $s=2$ or $t=2$, we use the fact that $\mathcal{P}$ is a weighted Katz-Tao $(\delta, t-\epsilon, A\delta^{-\epsilon})$-set and $\mathcal{T}$ is a weighted Katz-Tao $(\delta, s-\epsilon, B\delta^{-\epsilon})$-set for any $\epsilon\in (0,1)$. Thus an incidence estimate with $\delta^{-\epsilon}$-error can be got by using \eqref{incidence22}.
	\end{remark}

	\section{Proof of Theorem \ref{main} for $\mathbf{t\in(0,s]}$}\label{sec4}
	In this section, we prove the following proposition which implies Theorem \ref{main} when $t\in(0,s]$ due to the reduction in Subsection \ref{sub-2.2}. A subtle difference here is that the implicit constant in estimate \eqref{form-incidence2} is actually irrelevant to parameter "$t$". This is because the proof of Proposition \ref{pro-incidence2} uses the fact that $\mathcal{P}$ is also a Katz-Tao $(\delta,s)$-set if $t\leq s$.
	\begin{proposition}\label{pro-incidence2}
		Let $0\leq t\leq s\leq 1$ and $A, B\geq1$. Let $\mathcal{P}\subset\mathcal{D}_\delta$ be a Katz-Tao $(\delta, t, A)$-set. For each $q\in\mathcal{P}$, assume that there exists a Katz-Tao $(\delta, s, B)$-set
		\[\mathcal{F}(q)\subset\{p\in\mathcal{D}_\delta: p\cap \Gamma_q(\delta)\neq\emptyset\}.\]
		Write $\mathcal{F}:=\cup_{q\in\mathcal{P}} \mathcal{F}(q)$. Then there holds
		\begin{equation}\label{form-incidence2}
			\sum_{q\in\mathcal{P}}|\mathcal{F}(q)|\lesssim_{\psi, s} \log(\tfrac{1}{\delta})\sqrt{A B \delta^{-s}|\mathcal{P}||\mathcal{F}|}.
		\end{equation}
	\end{proposition}
	\begin{proof}
		First, applying Cauchy-Schwarz inequality gives
		\begin{equation*}
			\begin{split}
				\sum_{q\in\mathcal{P}}|\mathcal{F}(q)|&=\sum_{p\in\mathcal{F}}|\{q\in\mathcal{P}: p\in \mathcal{F}(q)\}|\\
				&\leq |\mathcal{F}|^{1/2}\big(\sum_{p\in\mathcal{F}}|\{(q_1, q_2)\in\mathcal{P}\times \mathcal{P}: p\in \mathcal{F}(q_1)\cap \mathcal{F}(q_2)\}|\big)^{1/2}\\
				&=|\mathcal{F}|^{1/2}\left(\sum_{q}|\mathcal{F}(q)|+\sum_{q_1\neq q_2}|\mathcal{F}(q_1)\cap \mathcal{F}(q_2)|\right)^{1/2}.
			\end{split}
		\end{equation*}
		If the diagonal sum ($q_1=q_2$) dominates, we get by the basic fact $|\mathcal{F}|\leq B\delta^{-s} |\mathcal{P}|$ that
		\[\sum_{q\in\mathcal{P}}|\mathcal{F}(q)|\lesssim |\mathcal{F}|=\sqrt{|\mathcal{F}||\mathcal{F}|}\leq \sqrt{B \delta^{-s}|\mathcal{P}||\mathcal{F}|}.\] 
		
		If the "off-diagonal" sum dominates, we have
		\begin{equation}\label{two sums}
			\begin{split}
				\sum_{q\in\mathcal{P}}|\mathcal{F}(q)| &\lesssim |\mathcal{F}|^{1/2}\left(\sum_{q_1\neq q_2}|\mathcal{F}(q_1)\cap \mathcal{F}(q_2)|\right)^{1/2}\\
				&=|\mathcal{F}|^{1/2}\left(\sum_{\substack{q_1\neq q_2\\ x_1=x_2}}|\mathcal{F}(q_1)\cap \mathcal{F}(q_2)|+ \sum_{\substack{q_1\neq q_2\\ x_1\neq x_2}}|\mathcal{F}(q_1)\cap \mathcal{F}(q_2)|\right)^{1/2}.
			\end{split}
		\end{equation}
		From now on, we use $(x_i,y_i)$ to denote the center points of $q_i\in\mathcal{P}$. To estimate the first sum, we claim that $|y_1-y_2|\lesssim_\psi \delta$ if $\Gamma_{q_1}(\delta)\cap \Gamma_{q_2}(\delta)\neq \emptyset$ with $q_1\neq q_2$ and $x_1=x_2$. Write $\psi_i(x):=\psi(x-x_i)+y_i$ with $i=1, 2$. Indeed, for any $(x,y)\in \Gamma_{q_1}(\delta)\cap \Gamma_{q_2}(\delta)$, there exist $(z_i,w_i)\in \Gamma_{q_i}$ such that 
		\[|(x,y)-(z_i,w_i)|\leq \delta, \quad i=1,~2.\]
		Then by triangle inequality we infer that
		\begin{equation}\label{form-3.1}
			\begin{split}
				|y_1-y_2|&=|\psi_1(x)-\psi_2(x)|\leq |\psi_1(x)-\psi_1(z_1)|+|\psi_1(z_1)-\psi_2(z_2)|+|\psi_2(x)-\psi_2(z_2)|\\
				&\leq C(\psi)|x-z_1|+|w_1-w_2|+C(\psi)|x-z_2|\lesssim_\psi \delta.
			\end{split}
		\end{equation}
		This means for a fixed $q_1$ there are at most $C(\psi)$ cubes $q_2$ such that $x_1=x_2$ and $\mathcal{F}(q_1)\cap \mathcal{F}(q_2)\neq\emptyset$. Hence we can deduce
		\begin{equation}\label{vertical}
			\sum_{\substack{q_1\neq q_2\\ x_1=x_2}}|\mathcal{F}(q_1)\cap \mathcal{F}(q_2)|=\sum_{q_1}\sum_{\substack{q_2\neq q_1\\ x_2=x_1}}|\mathcal{F}(q_1)\cap \mathcal{F}(q_2)|\lesssim_\psi B\delta^{-s}|\mathcal{P}|,
		\end{equation}
		using the fact $|\mathcal{F}(q)|\lesssim B \delta^{-s}$ by the Katz-Tao $(\delta, s,B)$-condition of $\mathcal{F}(q)$.
		
		It remains to deal with the second sum on the right side of \eqref{two sums}. To get a good estimate for $|\mathcal{F}(q_1)\cap \mathcal{F}(q_2)|$, we need to study the geometry of $\Gamma_{q_1}(\delta)\cap \Gamma_{q_2}(\delta)$ when $x_1\neq x_2$. Note that the translated graph is given by
		\[\Gamma_{q_i}=\{(x,y): y-y_i=\psi(x-x_i),~x\in [-1,1]\}.\]
		Let $(x,y)\in \Gamma_{q_1}\cap \Gamma_{q_2}$, then $(x,y)$ solves the equations 
		\begin{equation*}
			\left\{\begin{aligned}
				& y-y_1=\psi(x-x_1),\\
				& y-y_2=\psi(x-x_2).
			\end{aligned}\right.
		\end{equation*}
		By using the mean value theorem, we can obtain
		\begin{equation}\label{comparable}
			|y_1-y_2|=|\psi(x-x_1)-\psi(x-x_2)|\lesssim_\psi |x_1-x_2|.
		\end{equation}
		Our main claim is that
		\begin{equation}\label{main claim}
			\textup{diam}\left(\Gamma_{q_1}(\delta)\cap \Gamma_{q_2}(\delta)\right)\lesssim_\psi \frac{\delta}{|x_1-x_2|}.
		\end{equation}
		This will conclude the whole proof. Indeed, \eqref{main claim} implies that $\mathcal{F}(q_1)\cap \mathcal{F}(q_2)$ is contained in a ball with radius $\sim_\psi \tfrac{\delta}{|x_1-x_2|}$. Since $t\leq s$, $\mathcal{P}$ is also a Katz-Tao $(\delta, s, A)$-set. By using \eqref{comparable}, \eqref{main claim} and the Katz-Tao $(\delta, s)$-conditions of both $\mathcal{F}(q)$ and $\mathcal{P}$, we compute
		\begin{equation*}
			\begin{split}
				\sum_{\substack{q_1\neq q_2\\ x_1\neq x_2}}|\mathcal{F}(p_1)\cap \mathcal{F}(p_2)| &\lesssim_\psi \sum_{\substack{q_1\neq q_2\\ x_1\neq x_2}}\frac{B}{|x_1-x_2|^s}\sim B\sum_{q_1 \in \mathcal{P}}\sum_{k=1}^{\log (\tfrac{1}{\delta})}\sum_{\substack{q_2\in \mathcal{P}\\|x_1-x_2|\in [2^{-k}, 2^{1-k})}}\frac{1}{|x_1-x_2|^s}\\
				&\lesssim B\sum_{q_1 \in \mathcal{P}}\sum_{k=1}^{\log (\tfrac{1}{\delta})}2^{ks}\left|\Big\{q_2: |y_1-y_2|\lesssim_\psi|x_1-x_2|\in [2^{-k}, 2^{1-k})\Big\}\right|\\
				&\lesssim B\sum_{q_1 \in \mathcal{P}}\sum_{k=1}^{\log (\tfrac{1}{\delta})}2^{ks}|\mathcal{P}\cap B_{c(\psi)2^{-k}}|\lesssim_{\psi, s} \log(\tfrac{1}{\delta})AB\delta^{-s}|\mathcal{P}|,
			\end{split}
		\end{equation*}
		as desired.
		
		It remains to show \eqref{main claim}. Let $\pi_1$ be the projection onto $x$-axis and denote $\psi_i(x):=\psi(x-x_i)+y_i$ with $i=1, 2$. As an intermediate goal, we show that
		\begin{equation}\label{mainclaim2}
			\pi_1\left(\Gamma_{q_1}(\delta)\cap \Gamma_{q_2}(\delta)\right)\subset \{x: |\psi_1(x)-\psi_2(x)|\lesssim_{\psi} \delta\}.
		\end{equation}
		For any $(x,y)\in \Gamma_{q_1}(\delta)\cap \Gamma_{q_2}(\delta)$, there exist $(z_i,w_i)\in \Gamma_i$ such that 
		\[|(x,y)-(z_i,w_i)|\leq \delta, \quad i=1,~2.\]
		Then by using triangle inequality as \eqref{form-3.1} we can get $|\psi_1(x)-\psi_2(x)|\lesssim_\psi \delta$, which proves \eqref{mainclaim2}. If we write $G(x)=\psi_1(x)-\psi_2(x)$, then by assumption \eqref{curvature assumption}
		\begin{equation}\label{lowergrowth}
			|G'(x)|=|\psi_1'(x)-\psi_2'(x)|=|\psi'(x-x_1)-\psi'(x-x_2)|\gtrsim_\psi |x_1-x_2|.
		\end{equation}
		As a consequence, \eqref{main claim} follows by combining \eqref{mainclaim2} and \eqref{lowergrowth}. This also completes the proof of Proposition \ref{pro-incidence2}.
	\end{proof}

	\appendix
	\section{An incidence estimate}\label{apend A}
	In this appendix, we sketch the proof of Theorem \ref{thm-incidence1} which we restate as follows. This is a modification of the proof of \cite[Theorem 5.2]{MR4751206}, but for completeness we include here. Again, by using Remark \ref{rmk-endcase}, the estimate \eqref{form-sharp incidence} with a $\delta^{-\epsilon}$-error also holds for any $s\in[0,1]$ and $t\in [0,2]$ with $s+t\leq2$ .
	\begin{thm}\label{thm-sharp incidence}
		Let $s\in[0,1)$ and $t\in [0,2)$ such that $s+t<2$ and let $A, B\geq 1$. Assume $\mathcal{P}\subset \mathcal{D}_\delta$ is a Katz-Tao $(\delta, t, A)$-set, and for each $q\in \mathcal{P}$ there exists a Katz-Tao $(\delta, s, B)$-set $\mathcal{F}(q)\subset \{p\in \mathcal{D}_\delta: p\cap \Gamma_q(\delta)\neq \emptyset\}$. Write $\mathcal{F}:=\bigcup_{q\in \mathcal{P}} \mathcal{F}(q)$. Then
		\begin{equation}\label{form-sharp incidence}
			\sum_{q\in \mathcal{P}}|\mathcal{F}(q)|\lesssim_{\psi, s,t} \sqrt{\delta^{-1}AB |\mathcal{F}||\mathcal{P}|}.
		\end{equation}
	\end{thm}
	\begin{proof}
		\textbf{Step 1: initial reduction.} 
		Since $\psi$ is strictly convex, $\Gamma$ can be divided into the decreasing part $\Gamma^-$ and the increasing part $\Gamma^+$. For each $q\in \mathcal{P}$, let $\Gamma_q^-(\delta)$ and $\Gamma_q^+(\delta)$ be the $\delta$-neighborhood of $\Gamma_q^-:=q+\Gamma^-$ and $\Gamma_q^+:=q+\Gamma^+$ respectively, then it is easy to see 
		\[\Gamma_q(\delta)=\Gamma_q^-(\delta)\cup \Gamma_q^+(\delta).\]
		Moreover, we define
		\[\mathcal{F}^-(q):=\{p\in \mathcal{F}(q): p\cap \Gamma_q^-(\delta)\neq \emptyset\},\quad \mathcal{F}^+(q):=\{p\in \mathcal{F}(q): p\cap \Gamma_q^+(\delta)\neq \emptyset\},\]
		then either $\sum_q|\mathcal{F}(q)|\sim \sum_q|\mathcal{F}^-(q)|$ or $\sum_q|\mathcal{F}(q)|\sim \sum_q|\mathcal{F}^+(q)|$. Without loss of generality, we may assume the latter case and denote $\Gamma_q^+(\delta)$ and $\mathcal{F}^+(q)$ still by $\Gamma_q(\delta)$ and $\mathcal{F}(q)$ respectively.
		
		Moreover, we divide $\mathcal{F}$ into four sub-families, say $\mathcal{F}_{i,j}$ with $i, j\in\{0,1 \}$, where $\mathcal{F}_{i,j}$ is the collection of $\delta$-cubes in $\mathcal{F}$ with upper-right vertex $(m\delta, n\delta)$ satisfying $m\equiv i, n\equiv j~(\text{mod}~2)$. By translating the configuration if necessary, we may assume that every dyadic $\delta$-cube in $\mathcal{F}$ has upper-right vertex in $(\delta (2\mathbb Z+1))^2$.
		
		\textbf{Step 2: constructing less concentrated pockets.} Fix a dyadic number $\omega \in [\delta, 1]$, we aim to construct a set of dyadic $\delta$-cubes $\mathcal{F}_\omega$ contained in $[0,\omega]^2$. Let $\mathcal{C}$ be a standard $s$-dimensional Cantor set on $[0,\omega]$ and let $\mathcal{C}(\delta)$ be its $\delta$-neighborhood. Define $\mathcal{F}_\omega$ as the set of dyadic $\delta$-cubes with upper-right vertices $(m\delta, n\delta)$, where $m ,n \in\mathbb Z$, $1\leq m, n \lesssim \tfrac{\omega}{\delta}$ and either $m\delta$ or $n\delta$ belong to $\mathcal{C}(\delta)$. An easy observation shows that $|\mathcal{F}_\omega|\lesssim (\tfrac{\omega}{\delta})^{s+1}$ and for any positive number $d\in [\delta, \omega]$ we have:
		\begin{equation}\label{form-a1}
			\left|\{m: m\delta \in \mathcal{C}(\delta)\cap [0, d]\}\right|\gtrsim (\tfrac{d}{\delta})^{s}.
		\end{equation}
		
		\textbf{Step 3: fixing over-concentrated pockets.} We want to replace the over-concentrated pockets in $\mathcal{F}$ by less concentrated pockets constructed in \textbf{Step 2} so that $\mathcal{F}$ will be replaced by a weighted Katz-Tao $(\delta, s+1, O(B))$-set $\mathcal{F}'$. Fix $\omega\in [\delta,1]\cap 2^{-\mathbb N}$. Let $\mathcal{W}$ be the set of cubes in the family $\bigcup_{\delta\leq\omega\leq1} \mathcal{D}_\omega$ that contain $\geq B (\tfrac{\omega}{\delta})^{1+s}$ cubes in $\mathcal{F}$. Let $\mathcal{R}$ be the maximal elements of $\mathcal{W}$, which means any $p\in\mathcal{R}$ cannot be contained in other cubes in $\mathcal{W}$. It is clear that $\mathcal{R}$ is a disjoint family of dyadic cubes. Before constructing $\mathcal{F}'$, we need to verify the following technical lemma.
		\begin{sublemma}
			Let $\mathcal{F}_\omega$ be the set constructed in \textbf{Step 2}. For any $q\in \mathcal{P}$, we have 
			\begin{equation}\label{estimate11}
				B|\{p\in \mathcal{F}_\omega: p\cap \Gamma_q(\delta)\neq \emptyset\}|\gtrsim_\psi |\mathcal{F}(q)\cap [0,\omega]^2|.
			\end{equation}
		\end{sublemma}
		\begin{proof}
			Fix $q\in \mathcal{P}$, we may assume that there exists $p\in \mathcal{F}(q)\cap [0,\omega]^2\neq \emptyset$. Let $d$ be the length of the projection of $\Gamma_q(\delta)\cap [0,\omega]^2$ onto $x$-axis. Note that this projection is a consecutive interval due to the reduction in \textbf{Step 1}. We also assume that $\Gamma_q(\delta)$ intersects the left or right edge of $[0,\omega]^2$. Otherwise, if $\Gamma_q(\delta)$ intersects the bottom or top edge of $[0,\omega]^2$ we can instead consider projection of $\Gamma_q(\delta)\cap [0,\omega]^2$ onto $y$-axis.
			
			If $d\geq\delta$, then $\Gamma_q(\delta)$ intersects at least one cube in $\mathcal{F}_\omega$ with upper-right vertex $(m\delta, n\delta)$ for each $m\delta\in \mathcal{C}(\delta)\cap [0,d]$ (if $\Gamma_q(\delta)$ intersects the left edge) or $m\delta\in \mathcal{C}(\delta)\cap [\omega-d,\omega]$ (if $\Gamma_q(\delta)$ intersects the right edge). Thus we deduce by \eqref{form-a1}
			\[|\{p\in \mathcal{F}_\omega: p\cap \Gamma_q(\delta)\neq \emptyset\}|\gtrsim (\tfrac{d}{\delta})^s.\]
			If $d<\delta$ and $\Gamma_q(\delta)$ intersects the left edge of $[0,\omega]^2$, we claim that the upper-right vertex $(x_p,y_p)$ of $p$ equals $(\delta, n\delta)$. Otherwise $x_p\geq 3\delta$ since $(x_p,y_p)$ lies in $(\delta(2\mathbb Z+1))^2$. But this will cause $d>2\delta$, which is a contradiction. If $\Gamma_q(\delta)$ intersects the right edge of $[0,\omega]^2$, the same argument shows that $(x_p,y_p)$ equals $(\omega-\delta, n\delta)$ where $\omega-\delta\in \mathcal{C}(\delta)\cap \delta \mathbb Z$. Combining the two cases gives
			\[|\{p\in \mathcal{F}_\omega: p\cap \Gamma_q(\delta)\neq \emptyset\}|\gtrsim \max\{1, (\tfrac{d}{\delta})^s\}.\]
			On the other hand, by assumption \eqref{curvature assumption} and simple geometric argument, we see that $\Gamma_q(\delta)\cap [0,\omega]^2$ is contained in a ball with radius $\sim_\psi d$ and $\mathcal{F}(q)\cap [0,\omega]^2$ is contained in a ball with radius $\sim_\psi (d+\delta)$. Since $\mathcal{F}(q)$ is a Katz-Tao $(\delta, s, B)$-set, we infer
			\[|\mathcal{F}(q)\cap [0,\omega]^2|\lesssim_\psi B(\tfrac{d+\delta}{\delta})^s\lesssim B\max\{1, (\tfrac{d}{\delta})^s\}\lesssim B|\{p\in \mathcal{F}_\omega: p\cap \Gamma_q(\delta)\neq \emptyset\}|.\]
		\end{proof}
		We are now ready to construct $\mathcal{F}'$. For each $Q\in \mathcal{R}\cap\mathcal{D}_\omega$, let $\mathcal{F}'(Q)$ be a translation of $\mathcal{F}_\omega$ placed in $Q$, then define
		\[\mathcal{F}':=\{p\in \mathcal{F}: p\not\subset \cup \mathcal{R}\}\cup \bigsqcup_{Q\in \mathcal{R}}\mathcal{F}'(Q).\]
		Next, we associate $\mathcal{F}'$ with a weight function 
		\begin{equation*}
			w(p):=\begin{cases}
				B, & ~~\text{if}~p\in \bigsqcup_{Q\in \mathcal{R}}\mathcal{F}'(Q),\\
				1, & ~~\text{if}~p\in\{p\in \mathcal{F}: p\not\subset \cup \mathcal{R}\}.
			\end{cases}
		\end{equation*}
		Let $\mathcal{T}:=\{\Gamma_q(\delta): q\in\mathcal{P}\}$. Then the following properties are easy to verify:
		\begin{itemize}
			\item[(P1)] \phantomsection \label{P1} $\sum_{p\in\mathcal{F}'} w(p)\lesssim |\mathcal{F}|$ and $\sum_{q\in\mathcal{P}}|\mathcal{F}(q)|\lesssim \mathcal{I}_w(\mathcal{F}',\mathcal{T})$;
			\item[(P2)] \phantomsection \label{P2} $\mathcal{F}'$ is a weighted Katz-Tao $(\delta, s+1, O(B))$-set.
		\end{itemize}
		Indeed, note \eqref{estimate11} also holds if $[0,\omega]^2$ is replaced by any $Q\in\mathcal{D}_\omega$ and $\mathcal{F}_\omega$ is replaced by $\mathcal{F}'(Q)$, then for any $q\in\mathcal{P}$ and $Q\in \mathcal{R}$ we have
		\[\sum_{p\in \mathcal{F}'(Q)\cap \Gamma_q(\delta)}w(p)=B|\{p\in \mathcal{F}'(Q): p\cap \Gamma_q(\delta)\neq \emptyset\}|\gtrsim_\psi |\mathcal{F}(q)\cap Q|,\]
		which implies $\sum_{q\in\mathcal{P}}|\mathcal{F}(q)|\lesssim \mathcal{I}_w(\mathcal{F}',\mathcal{T})$. Moreover, from \textbf{Step 2} we know $\sum_{p\in \mathcal{F}'(Q)}w(p)\lesssim B(\tfrac{\omega}{\delta})^{1+s}\lesssim|\mathcal{F}\cap Q|$ for each $Q\in\mathcal{R}$, thus $\sum_{p\in\mathcal{F}'} w(p)\lesssim |\mathcal{F}|$ and proves \nref{P1}. Property \nref{P2} is clear by definition of $\mathcal{F}'$. Finally, we apply Corollary \ref{incidence11} to get (recall $s+1+t<3$)
		\[\sum_{q\in\mathcal{P}}|\mathcal{F}(q)|\lesssim \mathcal{I}_w(\mathcal{F}',\mathcal{T})\lesssim_{\psi,s,t} \sqrt{\delta^{-1}AB|\mathcal{T}|\sum_{p\in\mathcal{F}'} w(p)}\lesssim\sqrt{\delta^{-1}AB|\mathcal{P}||\mathcal{F}|}.\]
	\end{proof}

	\bibliographystyle{plain}
	\bibliography{references}
	
\end{document}